\documentclass[12pt]{article}
\usepackage{graphicx,amssymb,amsmath,amsthm,cite}
\newtheorem{proposition}{Proposition}
\newtheorem{property}{Property}
\newtheorem{remark}{Remark}

\def\be{\begin{equation}}
\def\ee{\end{equation}}
\def\ben{\begin{eqnarray}}
\def\een{\end{eqnarray}}

\newcommand{\la}{\langle}
\newcommand{\ra}{\rangle}

\newcommand{\til}{\tilde}

\newcommand{\wt}{{w}}

\newcommand{\wtkp}{{w}^{k+1}}
\newcommand{\wtik}{\wt_i^k}
\newcommand{\cik}{c_i^k}
\newcommand{\wtikp}{\wt_i^{k+1}}

\def\kj{k \setminus j}
\def\rj{r \setminus j}

\def\op{\hat{P}}
\def\oe{\hat{E}}
\def\S{{\cal{S}}}
\def\SS{{\cal{S}}}

\def\SH{{\cal{H}}}
\def\SV{{\cal{V}}}

\def\SWr{\til{{\cal{W}}}_r}
\def\SWo{\til{{\cal{W}}}_0}
\def\SVr{\til{{\cal{V}}}_r}
\def\SVo{\til{{\cal{V}}}_0}
\def\SVK{{\cal{V}}_k}

\def\SWKP{{\cal{W}}_{k+1}}
\def\SVKJ{{\cal{V}}_{\kj}}
\def\SWKJ{{\cal{W}}_{\kj}}
\def\SWKrJ{{\cal{W}}_{\rj}}
\def\SW{{\cal{W}}}
\def\SWK{{\cal{W}}_k}
\def\SWKr{{\cal{W}}_r}
\def\SWC{{\cal{W}^\bot}}
\def\SVC{{\cal{V}^\bot}}
\def\EVW{\hat{E}_{\SV \SWC}}
\def\EVWr{\hat{{E}}_{\SVr \SWr}}
\def\EVV{\hat{E}_{\SV \SVC}}
\def\EVKW{\hat{E}_{\SV_k \SWC}}
\def\EVWKP{\hat{E}_{\SV_{k+1} \SWC}}

\DeclareMathOperator{\Span}{span}

\def\C{\mathbb{C}}
\def\emptyy{\{0\}}

\def\bmth#1{{\mbox{\boldmath$#1$}}}
\def\ei{{\bmth{e}}_i}
\newcommand{\ut}{w}

\newcommand{\utik}{w_i^k}

\title{Sparse Representations for Structured Noise Filtering}
\author{Bishnu P. Lamichhane\footnotemark[1]\quad and
Laura Rebollo-Neira\thanks{Aston University,
Aston Triangle, B4 7ET, UK, {\tt b.p.lamichhane@aston.ac.uk,
rebollol@aston.ac.uk }}}
\begin{document}

\maketitle
%\begin{abstract}
%\end{abstract}
%{\bf Key words}. Oblique projection, Oblique Matching Pursuit, 
%Greedy Algorithms, Adaptive subset selection, signals separation.\\
%
%{\bf AMS subject classification}. 41A10, 33F05.

\abstract{The role of sparse representations in the context of structured 
noise filtering is discussed. A strategy, especially conceived  so as  to
address problems of an ill posed nature, is  presented. The proposed 
approach revises and extends the Oblique Matching Pursuit technique.
It is shown that, by working with an 
orthogonal projection of the signal to be filtered, it is possible 
to apply orthogonal matching  pursuit like strategies 
in order to accomplish the required signal discrimination.}

\section{Introduction}
The problem of structured noise filtering 
is introduced in \cite{BS94}, where a number of relevant 
signal processing applications are discussed. It can be 
posed as follows: consider that a signal $f$, represented 
as an element of an inner product space $\SH$, is produced 
by the superposition of two components, $f_1$ and $f_2$, 
each of which belongs to a different subspace of $\SH$. More 
precisely, $f=f_1 + f_2$ with $f_1 \in \S_1 \subset \SH$ and 
$f_2 \in \S_2 \subset \SH$. Structured noise filtering 
(to be also termed signal discrimination or signal splitting)
consists of singling out a particular component from the signal $f$. 
Provided that $\S_1$ and $\S_2$ are given,  such that 
$\S_1 \cap \S_2 = \emptyy$,  one component, say $f_1$, can 
be extracted from $f$  by an oblique projection 
onto $\S_1$ and along $\S_2$. On the contrary, the situation 
$\S_1 \cap \S_2 \ne \emptyy$ implies that the signal decomposition is 
not unique and the splitting  can not be tackled 
in a straightforward manner by
oblique projections. Moreover, 
even when theoretically the condition $\S_1 \cap \S_2 = \emptyy$ is 
satisfied, if the subspaces $\S_1$ and $\S_2$ are not well separated, 
the construction of the corresponding projector 
becomes ill posed.  Consequently, the signal splitting can not be achieved 
by numerical calculations in finite precision arithmetics. 
Here we focus on such a situation. {\em{We assume that the given
subspaces $\S_1$ and $\S_2$ are theoretically disjoint,  but
close enough to yield an ill posed problem}}.

Our proposal for the numerical realization of the signal splitting 
is focussed on the search of a subspace of the given
$\S_1$, where a class of signals is considered to lie.  It will be 
assumed throughout the paper that the class of signals to be 
considered is $K$-sparse in a spanning set for
$\S_1$. By this we mean that given a spanning set for $\S_1$ 
the corresponding linear superposition of a signal  
 has at most $K$ nonzero coefficients. The $K$-value should 
be less than or equal to the dimension of the subspace $\S_r \subset  
\S_1$ for which the construction of an oblique projection onto 
itself, and along $\S_2$, is well conditioned. This assumption 
is quite realistic, considering that in practice there is often  a
lack of complete knowledge on the actual subspace 
$\S_1$ and to be on the safe side one may overestimate it. 
 
The main motivation of this paper is to highlight the 
essential role that sparse representations play in the problem 
of structured noise filtering. Such representations have been 
the subject of considerable work over the last ten years
\cite{DT96,DMA97,Dev98,CDS98,Tem99,AHSE01,GMS03,Tro04,
GN04,Don06,CR06,CT06,GFV06}. We will dedicate special 
attention to discuss and illustrate `why' and `how' sparse representations 
are relevant in the present context.

A technique, termed Oblique Matching Pursuit 
(OBMP), has been recently advanced in relation to the above described 
problem \cite{Reb07b}. Such a technique evolves by stepwise 
selection of the sought subspace. The selection criterion is 
based on the consistency principle \cite{UA94,Eld03}. 
In this communication we revise and extend the OBMP technique. 
We show that by working with a particular
projection of the signal at hand, rather than with the signal itself, 
one can make use of previously proposed orthogonal 
matching pursuit like methodologies,
so as to look for the signal subspace yielding the correct splitting. 

The paper is organized as follows: In Section \ref{sec2} we 
introduce the mathematical setting for signal representation to 
be adopted here,  together with a discussion on the construction of 
oblique projections. Section \ref{sec3}  highlights the  importance of
the search for sparse solutions in the construction of oblique 
projectors for structured noise filtering. The proposed strategy 
is discussed in Section \ref{sec4}. The conclusions are presented in 
 Section \ref{sec5}. 
                                        
\section{Mathematical Framework}
\label{sec2}
We consider a signal, $f$, to be an element of an inner product space 
$\SH$. The square norm $||f||^2$ is then induced by the 
inner product that we indicate as $\la f, f \ra$ and is 
defined in such a way that  if  $a$  is a  number,  
$\la af, f \ra= a^\ast\la f, f \ra$, with $a^\ast$ representing 
the complex conjugate of $a$. For
 the purpose of this contribution we  assume 
that all the signals of interest belong to some finite 
dimensional subspace $\SV$ of $\SH$. Thus, there exists a 
finite set  $\{v_i \in \SH\}_{i=1}^M$  spanning $\SV$.  Consequently,   
for every signal in $\SV$ there is a set of numbers 
$\{c_i\}_{i=1}^M$  which allows us to 
express the signal as the linear superposition 
$$f= \sum_{i=1}^M c_i v_i,$$ 
which is also called {\em atomic decomposition}.

Although a signal was defined as an element of an abstract 
inner product space, for processing  purposes we need a numerical 
representation of such an object. The process of transforming a 
signal into a number is refereed to as  {\em {measurement}}
or {\em {sampling}}. 
The  mathematical operation performing  such a transformation is  then a
{\em {functional}}. Since considerations will be restricted 
to linear measurements,
we represent them by {\em {linear functionals}}. 
Thus,  making use of Riesz theorem \cite{RS80} we can  express a linear 
measurement as
$m= \la  w , f \ra $ for some $w \in \SH.$
Considering now $M$ measurements $m_i,\, i=1,\ldots,M$,
each of which is obtained by a measurement vector $w_i$, we have a numerical 
representation of $f$ as given by
\be
m_i= \la  w_i , f \ra,\quad i=1,\ldots,M.
\ee
The question concerning the possibility of reconstructing 
$f\in \SV$ from measurements obtained with 
vectors in a different subspace has been 
addressed in \cite{UA94,UZ97,Eld03,Hu07}.
It is in principle obvious that every signal in $\SV$ can be reconstructed
from  vectors $\{w_i \in \SH\}_{i=1}^M$ spanning a 
subspace $\SW \subset \SH$, provided that those vectors give rise
to a representation of any projector onto $\SV$. The difference 
in using one projector or another appears  when the projector 
acts on signals outside a subspace.
We summarize next some features relevant to the construction 
of projectors.
\subsection{Oblique projectors}
\label{op}
Every idempotent operator is a projector. Hence, an operator 
$\oe$ is a projector if
$\oe^2 =\oe$. The projection is
along its null space and onto its range.
When these subspaces
are orthogonal $\oe$ is called an
orthogonal projector, which is the case if and only if
$\oe$ is self-adjoint. Otherwise it is called {\em {oblique projector}}. 

Given two closed subspaces,
$\SV \in \SH$ and $\SWC \in \SH$,
such that $\SS=\SV + \SWC$ and $\SV \cap \SWC = \emptyy$,
the oblique projector operator onto $\SV$
along $\SWC$ will be represented as
$\EVW$. Then $\EVW$ satisfies
$\EVW^2= \EVW$
and, consequently, 
\ben
\EVW f&=& f, \quad \text{if} \quad f \in \SV  \nonumber\\
\EVW f&=& 0, \quad \text{if} \quad f \in \SWC. \nonumber
\een
In the particular case for which $\SWC =\SV^\bot$ the operator 
$\EVV$ is an orthogonal projection onto $\SV$. For indicating  an orthogonal 
projector onto a subspace, ${\cal{X}}$ say, we  use the 
particular notation $\op_{\cal{X}}$.

Consider that $\{v_i\}_{i=1}^M$ is a spanning set for 
$\SV$ and $\{u_i\}_{i=1}^M$ a spanning set for $\SW$, which is 
the orthogonal complement of $\SWC$ in $\SS$, i.e., 
$\SS= \SW \oplus \SWC$, with $\oplus$ 
indicating the orthogonal sum. Thus the spanning sets of 
$\SV$ and $\SW$  satisfy $\{u_i\}_{i=1}^M= \{\op_{\SW} v_i\}_{i=1}^M$. 
If the condition 
$\SV \cap \SWC = \emptyy$ is fulfilled, the operator 
$\EVW$ can be constructed as
\be
\label{evw}
\EVW=\sum_{i=1}^M v_i \la  w_i , \cdot \ra,
\ee
where the operation $\la w_i , \cdot \ra$ indicates that $\EVW$ 
acts by performing inner products. The vectors $\{w_i\}_{i=1}^M$ 
in \eqref{evw} are 
obtained from vectors $\{u_i\}_{i=1}^M$ through the equation 
\cite{Eld03,Reb07a}
\be
\label{wdu}
w_i= \sum_{j=1}^M {g}^\dagger_{i,j} u_j
\ee
with ${g}^\dagger_{i,j}$ the element $(i,j)$
of a matrix $G^\dagger$, a pseudo inverse of the matrix $G$
the elements of which are given by the inner products  
$\la u_i,v_j\ra, i,j=1,\dots,M$.
\begin{remark} The pseudo inverse allows for the possibility  of the 
spanning sets of $\SV$ and $\SW$ being redundant. However, the condition
$\SV \cap \SWC = \emptyy$ implies that $\SV$ and $\SW$ should have 
the same dimension and therefore the rank of $G$ equals the 
dimension of $\SV$ and $\SW$. 
\end{remark} 
For later convenience, we introduce at this point an 
alternative representation of $\EVW$. To this end, 
denoting as
$\ei, i=1,\ldots,M$ the standard orthonormal basis in
$\C^M$, we
%i.e., the Euclidean inner
%product $\la \ei , {\bmth{e}_j}\ra=\delta_{i,j}$,we 
define the operators $\hat{V}: \C^M \to \SV$ and $\hat{W}:\C^M \to \SW$
as
$$\hat{V}=\sum_{i=1}^M v_i \la \ei , \cdot \ra,
\;\;\;\;\;\;\;\;\;\;\;\;\;\;
\hat{W}=\sum_{i=1}^M  u_i \la  \ei, \cdot \ra.$$
Thus the corresponding adjoint operators  $\hat{W}^\ast: \SW \to \C^M $ and
$\hat{V} ^\ast: \SV \to \C^M$ are
$$\hat{V}^\ast=\sum_{i=1}^M \ei \la v_i , \cdot \ra,
\;\;\;\;\;\;\;\;\;\;\;\;\;\
\hat{W}^\ast=\sum_{i=1}^M \ei \la u_i , \cdot \ra.$$ 
It follows from the above definitions that 
$\op_{\SW} \hat{V}= \hat{W}$ and
$\hat{W}^\ast \op_{\SW} = \hat{W}^\ast.$

Considering that $\psi_n \in \C^M, \,n=1,\ldots,M$, 
are the eigenvectors of matrix $G=\hat{W}^\ast \hat{V}=\hat{W}^\ast \hat{W}$,
and assuming that there exist $N$ nonzero eigenvalues
$\lambda_n,\,n=1,\ldots,N$,  on ordering these eigenvalues 
in descending order we can express the matrix elements of the
Moore-Penrose pseudo inverse of $G$ 
as:%check the conjugates
\be
{g}^\dagger_{i,j} = \sum_{n=1}^N \psi_n(i) \frac{1}{\lambda_n} \psi_n^\ast(j), 
\ee
with $\psi_n(i)$ the $i$-th component of $\psi_n$. 
Moreover, the orthonormal vectors 
\be
\label{xi}
\xi_n= \frac{\hat{W}\psi_n}{\sigma_n},\quad \sigma_n=\sqrt{\lambda_n},\quad
n=1,\ldots,N
\ee
are singular vectors of $\hat{W}$, which satisfies 
$\hat{W}^\ast \xi_n=\sigma_n \psi_n$, as it is immediate to verify.
By defining now the vectors $\eta_n,\,n=1,\ldots,N$ as
\be
\label{eta}
\eta_n=  \frac{\hat{V}{\psi_n}}{\sigma_n},
\quad \,n=1,\ldots,N,
\ee
the projector $\EVW$ in \eqref{evw} is recast in the fashion 
\be
\label{evw2}
\EVW=\sum_{n=1}^N \eta_n \la \xi_n , \cdot \ra.
\ee
Inversely, the representation  \eqref{evw} of  $\EVW$  arises from 
\eqref{evw2}, since 
\be
\label{wdu2}
w_i= \sum_{n=1}^N  \xi_n \frac{1} {\sigma_n} \psi^\ast_n(i), \quad i=1,\ldots,M.
\ee
%Accordingly,  vectors $w_i$ in 
%\eqref{wdu} can are computed as
%\be
%\label{wdu2}
%w_i= \sum_{n=1}^N \psi_n^\ast(i) \frac{1}{\sigma_n} \phi_n(j) u_j.
%\ee
%By defining now the new vectors
\begin{proposition} 
\label{bio}
The vectors $\xi_n \in \SW,\,n=1,\ldots,N$ and  
$\eta_n \in \SV,\,n=1,\ldots,N$ 
given in \eqref{xi} and \eqref{eta}
are biorthogonal to each other and span $\SW$ and $\SV$, respectively.
\end{proposition}
\begin{proof}
Using \eqref{xi} and \eqref{eta} we have
\be
\la \xi_m ,\eta_n\ra = \frac{1}{\sigma_n \sigma_m} \la 
\hat{W}\psi_n, \hat{V}\psi_m\ra = \frac{1}{\sigma_n \sigma_m} 
\la\psi_n, \hat{W}^\ast\hat{V}\psi_m\ra =\delta_{n,m} 
\frac{\lambda_m}{\sigma_n \sigma_m}= \delta_{n,m},
\ee
which proves the biorthogonality property. 

The proof that $\Span\{\xi_n\}_{n=1}^N=\SW$ stems from  the fact that 
$\SW= \Span\{u_i\}_{i=1}^M = \Span\{w_i\}_{i=1}^M$, 
which allows us to express an arbitrary 
$g \in \SW$ as the linear 
combination
$g=\sum_{i=1}^M a_i w_i$. Then, using \eqref{wdu2}, we have
$g= \sum_{n=1}^N \til{a}_n \xi_n$ with $\til{a}_n= \frac{1}{\sigma_n}
\sum_{i=1}^M a_i \psi_n^\ast(i)$, which proves that 
$\SW \subset \Span \{\xi_i\}_{i=1}^N$. On the other hand 
for $g \in \Span \{\xi_i\}_{i=1}^N$ we can write $g= \sum_{n=1}^N 
d_n \xi_n$ and using \eqref{xi} we have $f=\sum_{i=1}^M \til{d}_i u_i $, 
with $\til{d}_i=  \frac{1}{\sigma_n}\sum_{n=1}^N d_n \psi_n(i)$. 
This proves that $\Span \{\xi_i\}_{i=1}^N \subset \SW$ 
and therefore $\Span \{\xi_n\}_{n=1}^N=\SW$.
The proof that $\Span \{\eta_n\}_{n=1}^N=\SV$ is equivalent 
to the previous one.
\end{proof} 
Since $\EVW f=f$ for every signal $f \in\SV$, regardless of the 
subspace $\SWC$, one can consider a different subspace $\SWC$ to 
construct measurement vectors and reconstruct a signal 
in $\SV$ using any set of such vectors, 
as long as  $\SV \cap \SWC = \emptyy$ with 
$\SWC$ the orthogonal complement of $\SW=\Span \{w_i\}_{i=1}^M$.
On the other hand, if $f \notin \SV$,  the
measurement vectors can be chosen to be suitable for the
particular processing task. For instance, if the goal is to produce an 
approximation $f_\SV$ of  $f$ in the subspace 
$\SV$, then in oder to minimize the distance $||f - f_\SV||$ 
 we  need  $f_\SV= \op_{\SV}f$.  Any other 
projection would yield a distance $||f -  \EVW f||$ which 
satisfies \cite{UA94}
$$ || f - \op_\SV f|| \le ||f -  \EVW f|| \le \frac{1}{\cos(\theta)} 
|| f - \op_\SV f||,$$
where $\theta$ is the minimum angle between the subspaces $\SV$ 
and $\SW$.  The equality is attained for $\SV= \SW$,
which correspond to the orthogonal projection.

However, if the aim were to discriminate from a 
signal produced by different phenomena only the component in 
$\SV$, then, as discussed below, an oblique projection 
turns to be appropriate.

Suppose that a signal $f$ is the superposition of two 
signals $f_1$ and $f_2$ with $f_1 \in \SV$ and $f_2$ in 
$\SWC$. The projection that will rescue $f_1$ 
from $f$ is $\EVW f$. A number of signal processing 
examples where an oblique projection is required are
given in \cite{BS94}. Provided that the subspaces hosting  the signal 
components are well separated, the discrimination of components with 
different structure is successful. Unfortunately, this is not 
 always the case and the construction of the necessary projector 
may generate an ill posed problem.

\section{The need for sparse representations in the present context}
\label{sec3}
This section is dedicated to illustrate, by recourse to
 a numerical example, 
the crucial role that the search for   
sparse solutions plays in the construction of oblique 
projectors for signal discrimination.  Consider that 
the spaces $\SV$ and $\SWC$, such that $\SV \cap \SWC = \emptyy$, 
are given, and the spanning set for $\SV$ is a basis of dimension $M$.
For constructing the dual vectors $w_i$ as in 
\eqref{wdu2} we first construct 
$\op_{\SWC}$, to generate the vectors 
\be
\label{defu}
u_i = v_i -   \op_{\SWC} v_i, \quad i=1,\ldots,M
\ee
spanning $\SW$.
\begin{remark}
Since $\SV \cap \SWC = \emptyy$,  and the set $\{ v_i\}_{i=1}^M$ is 
assumed to be
linearly independent,  the set $\{ u_i\}_{i=1}^M$ is also
linearly independent. 
\end{remark}
\begin{proof}
Suppose that $\sum_{i=1}^M b_i u_i=0$ for some set of numbers 
$\{b_i\}_{i=1}^M$ .
Then \eqref{defu} implies $g= \op_{\SWC} g$ for
$g= \sum_{i=1}^M b_i v_i$. Since by definition $g \in \SV$,
and $\SV \cap \SWC = \emptyy$ by hypothesis,
we conclude that $g=0$. Hence, the fact that the set
$\{v_i\}_{i=1}^M$ is linearly
independent implies that $b_i=0,\,i=1,\ldots,M$, which establishes that the
set $\{u_i\}_{i=1}^M$ is linearly independent.
\end{proof}
\begin{remark}
Conversely, the fact that nonzero vectors constructed as in
\eqref{defu} are linearly independent implies that $\SV \cap \SWC = \emptyy$
\cite{Reb07b}.
\end{remark}
%\begin{remark}
%From the previous remarks it follows that
%$\SV \cap \SWC =\emptyy$ implies that
%$\SV+ \SWC =\SW \oplus \SWC$ 
%with $\oplus$ indicating the orthogonal sum and $+$ the
%direct sum.
%\end{remark}
In order to render the numerical calculation of the 
vectors $w_i, \,i=1,\ldots,M$ spanning $\SW$ as stable as possible, 
it is  convenient to orthogonormalize vectors $u_i,\,i=1,\ldots,M$  to 
obtain the vectors $q_i,\,i=1,\ldots,M$ satisfying 
$\la q_i , q_j \ra=\delta_{i,j}$.
With these vectors we construct the $M\times M$ matrix $G$ having elements 
$\la q_i , v_j \ra$.  For the situation considered here this matrix
has an inverse. Let us denote the element $(i,j)$  of $G^{-1}$ as
$g^{-1}_{i,j}$ and  construct the corresponding vectors 
$w_i,\, i=1,\ldots,M$ as prescribed in \eqref{wdu} or \eqref{wdu2}. 
As will be illustrated by the numerical example below, in spite of the 
fact that `theoretically' $\SV \cap \SWC = \emptyy$, numerical errors,
due to the existence of small singular values, may cause the 
failure to find the unique signal splitting that theoretically 
one should expect.

{\bf{Example 1.}} Let $\SV$ be the cardinal cubic spline 
space with distance $0.065$
between consecutive knots, on the interval $[0,10]$. This is 
a subspace of dimension $M=163$, which we span using a
B-spline basis $\{B_i(x),\,x\in [0,10]\}_{i=1}^{163}$.
The background we wish to filter belongs
to the subspace $\SWC$ spanned by the set of functions
$\{y_i(x)=(x+1)^{-0.05i},\,x\in [0, 10]\}_{i=1}^{50}$. 
Here the inner product is defined  as 
$\la f, g\ra=\int_{0}^{10} f(x)^\ast g(x)\,dx$, and all 
the integrals are computed numerically. 

This example is very illustrative of how sensitive to 
numerical errors the computation of oblique 
projectors is.
The subspace we are dealing with are disjoint: the 
last five singular values of the corresponding 
matrix $G$ are: 
$$0.2305,  0.2298, 9.3211 \times 10^{-4}, 
2.5829 \times 10^{-6}, 2.5673 \times 10^{-7},$$
while the first is $\sigma_1=1.5018$. 
The smallest singular value  cannot be considered a 
numerical representation of zero when the calculations are being 
carried out in double precision arithmetic. Hence, one can assert that 
the condition $\SV \cap \SWC = \emptyy$ is fulfilled. 
However, due to the three small singular values
the computation of the measurement vectors in the whole subspace $\SW$  
is inaccurate enough to cause the failure to correctly separate signals  
in $\SV$ from their background. The    
left graph of Figure 1 
is generated by a random superposition or $K=70$ B-splines added to 
a background in the given $\SWC$.
The broken line in the right graph  represents  
the oblique projection onto the given $\SV$ 
along $\SWC$. As can be seen, the projection 
does not produce the required signal, which is 
represented by the continuous dark line in the same graph.   
Now, since the spectrum of singular values has a clear jump 
(the last three singular values are far from the previous ones)
it might seem that one could regularize the calculation by truncation 
of singular values. Nevertheless, such a methodology turns out to be not
appropriate for the present problem, as it does not yield the correct 
separation. The light lines in the 
right graph of Figure 1 depict the three approximations 
obtained by neglecting  one, two and three singular values. 

Propositions \ref{tru1} below  analyzes the effect that regularization 
by truncation  of singular values produces in the resulting 
 projection. 

\begin{figure}[!ht]
\label{f1}
\begin{center}
\includegraphics[width=8cm]{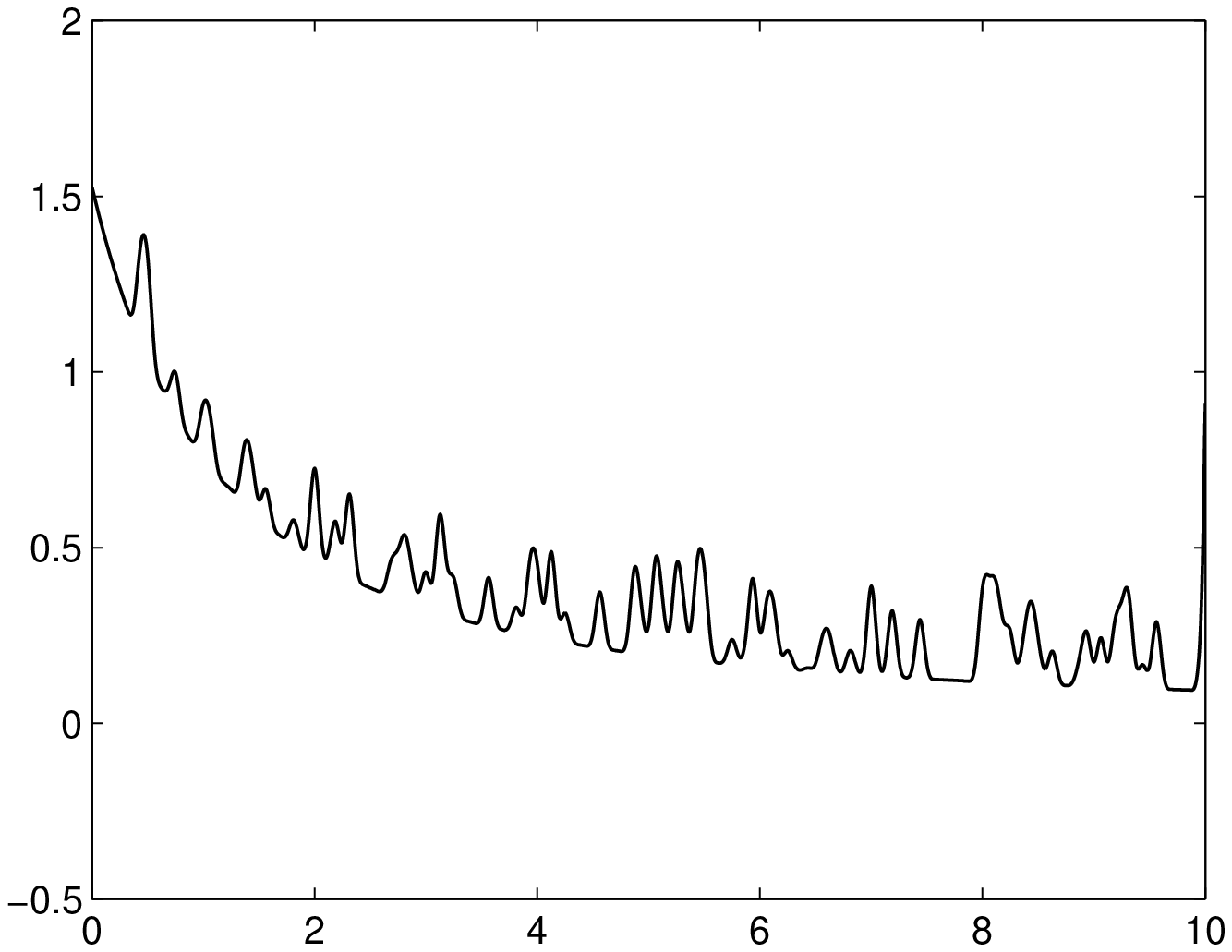}
\includegraphics[width=8cm]{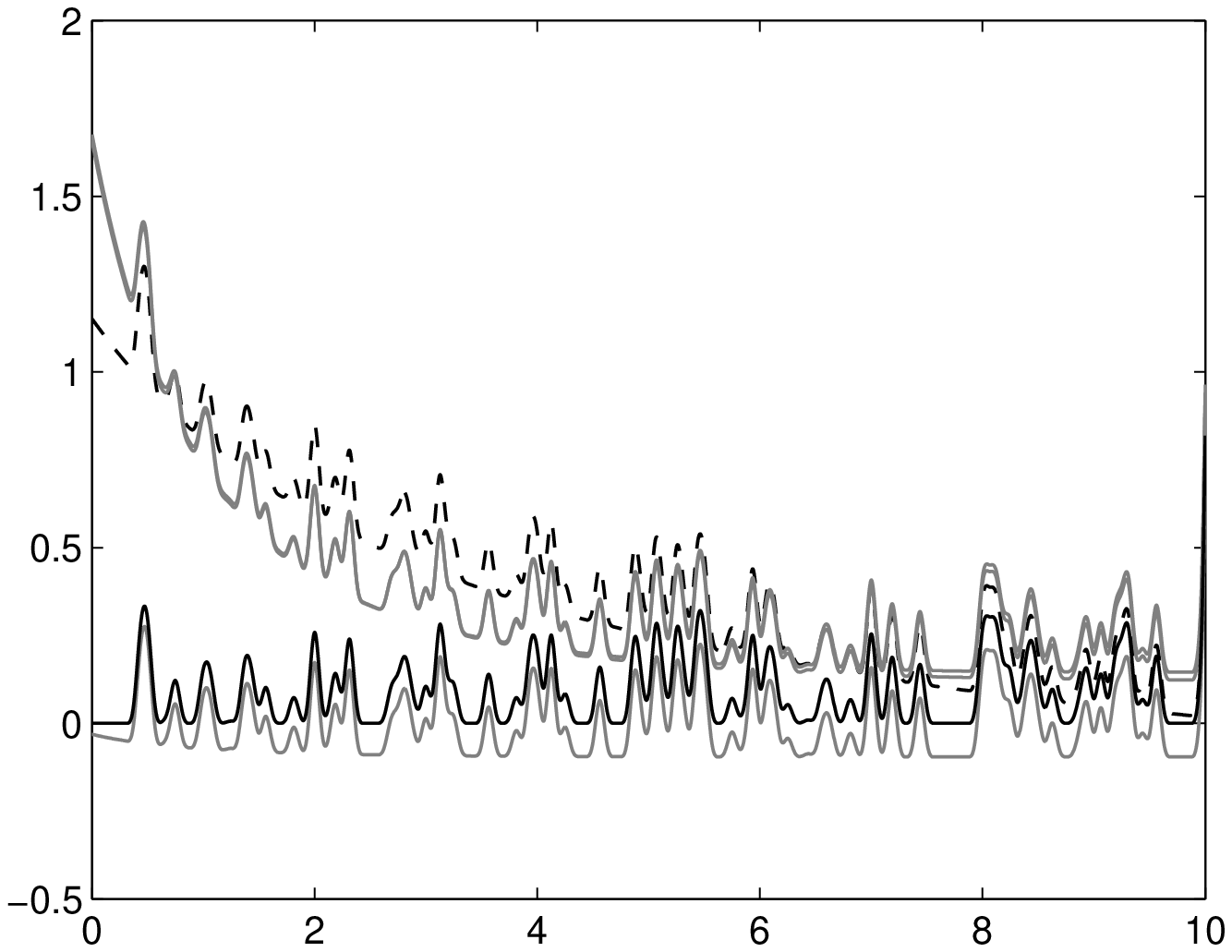}
\end{center}
\caption{Left graph: signal plus background.
Right graph: the dark continuous line corresponds to the signal to be 
discriminated from the one in the left graph. The broken line corresponds
to the approximation resulting from the oblique projection. The light 
lines correspond to the approximations obtained  by  
truncation of singular values (the one closest to the required 
signal correspond to the truncation of three singular  values).}
\end{figure}

\begin{proposition}
\label{tru1}
Truncation of the expansion \eqref{evw2} to consider up to $r$ terms, 
produces an oblique projector 
along $\SWr=\SWC + \SWo + \SVo$, with $\SVo= \Span\{\eta_i\}_{i=r+1}^N$
and $\SWo= \Span\{\xi_i\}_{i=r+1}^N$, onto $\SVr = \Span\{\eta_i\}_{i=1}^r$. 
\end{proposition}
\begin{proof}
The biorthogonality between  $\{\xi\}_{i=1}^r $ and  
$\{\eta_i\}_{i=1}^r$ established in
Proposition \ref{bio} ensures that 
$\EVWr=\sum_{i=1}^r \eta_i \la \xi_i , \cdot \ra$ is a projector, since
$\EVWr^2=\EVWr$. 

As established in Proposition \ref{bio}, 
 $\SV=\Span\{\eta_i\}_{i=1}^N$, 
and therefore every $f\in \SV$  can be 
decomposed as $f= f_r + f_o$ with $f_r \in \Span\{\eta_i\}_{i=1}^r$ and 
$f_o \in \Span\{\eta_i\}_{i=r+1}^N$. Moreover, $\EVWr f= f_r, 
\EVWr f_r= f_r$, and $\EVWr f_o= 0$, which proves that the 
projection is onto $\SVr$ and $\SVo$ is included in the 
null space of $\EVWr$. Equivalently, for  every 
$g_o \in \SWo=\Span\{\xi_i\}_{i=r+1}^N$ we have
$\EVWr g_o=0$, because the set $\{\xi_i\}_{i=1}^N$ is orthonormal. 
Thus, $\SWo$ is included in the null space of $\EVWr$.
\end{proof}

\subsection{Getting ready for a greedy search of the sparse solution} 
We discuss here the properties that will be of assistance in the next 
section, where we will 
present our strategy for the search of the sparse representation 
achieving the desired signal discrimination.
The goal is to avoid the computation of the measurement vectors 
in the whole subspace.
Instead, we strive to find the subspace $\SV_K \subset \SV$, where the
signal component one wants to  discriminate from the noise
is assumed to lie.  
We  work under the hypothesis that the subspace $\SWC$ is given and
fixed. Furthermore, $\SV \cap \SWC = \emptyy$,
which implies that
there exists a unique solution for the signal splitting.
%$\SVKP= \SVK + v_{k+1}=
%\Span\{v_i\}_{i=1}^{k+1}$
The problem 
we need to address arises from the fact that, if 
the subspaces $\SV$ and $\SWC$ are not well separated,  
the numerical calculation of the  measurement vectors 
is not accurate (due to the numerical operations being carried 
out in finite precision arithmetic). As a consequence, the  
representation of the corresponding projector fails to produce the 
correct signals separation. 

Assuming that we are able to accurately compute in finite precision arithmetic
 $r$  measurement vectors,   
we could attempt to filter structured 
noise of a signal belonging to a subspace spanned by at most $r$ of 
such vectors (i.e. the expansion of the signal in $\SV$ should have at most 
$r$ nonzero coefficients). However, even possessing this knowledge about 
the signal, the problem of finding the right subspace would be in general
intractable: out of a set of cardinality $M$ there exist $\tbinom{M}{r}$   
possible subsets of cardinality $r$.
An adaptive  strategy for the
subspace selection, given a signal, is advanced in \cite{Reb07b}. 
Before revising and extending that strategy  we need to recall two 
relevant properties of oblique projectors.
\begin{property}
\label{pro1}
The oblique projector $\EVW$ satisfies
$\op_{\SW} \EVW= \op_{\SW}$.
\end{property}
\begin{proof} 
It readily follows by applying $\op_{\SW}$ on both sides of 
\eqref{evw} or \eqref{evw2}. Since $\la u_i,v_j\ra = \la u_i, u_j\ra$, 
one has
\be
\label{pw}
\op_{\SW} \EVW= \sum_{i=1}^M u_i \la w_i , \cdot \ra= \op_{\SW}.
\ee
Moreover, since $\op_{\SW} \eta_i= \xi_i$, considering \eqref{evw2} 
we have
\be
\label{pw2}
\op_{\SW} \EVW = \sum_{i=1}^N \xi_i \la \xi_i , \cdot \ra= \op_{\SW}.
\ee
\end{proof}
\begin{property} 
\label{orcom}
Given a signal $f$ in 
$\SV + \SWC= \SW \oplus\SWC$, the only vector $g\in \SV$
satisfying 
\be
\label{co2}
\op_{\SW}f= \op_{\SW}g
\ee
is $g= \EVW f$.
\end{property}
\begin{proof}
If $g= \EVW f$ \eqref{co2} trivially follows from Property \ref{pro1}.
Let us assume  now that there
exists $g \in \SV$ such that  \eqref{co2} holds.
Then $\op_{\SW} (f-g)=0$, i.e.,
$(f-g)\in \SWC$. Hence $ \EVW(f-g)=0$ and,
since  $g \in \SV$,  this implies that
$\EVW f=g$.
\end{proof}
Let us suppose that $\SVK= \Span\{v_i\}_{i=1}^k$ is given and 
the spanning set is linearly independent. Assuming that
$\SVK \cap \SWC = \emptyy$ we guarantee that the set of 
vectors $\{u_i\}_{i=1}^k$, with $u_i$ given in \eqref{defu} is also linearly 
independent. Therefore the dimension of $\SVK$ is 
equal to the dimension of $\SWK= \Span \{u_i\}_{i=1}^k= \Span{
\{\wtik\}_{i=1}^k}$. 
We use now a superscript $k$ to indicate that the measurement 
vectors $\{\wtik\}_{i=1}^k$  span $\SWK$. 
Hence these vectors give rise to  
the oblique projection of a signal $f$, onto $\SVK$ and
along $\SWC$, as given by:
\be
\label{ato}
\EVKW f =\sum_{i=1}^k v_i \la  \wtik , f \ra = 
\sum_{i=1}^k \cik v_i. 
\ee
It is clear from \eqref{ato} that if the atoms in the  atomic 
decomposition were to be changed (or some  
atoms  were added to or deleted from the decomposition) 
the measurement vectors $\wtik$, and consequently the coefficients $\cik$
in \eqref{ato}, would need to be modified. The recursive equations 
below provide an effective way of implementing the task.\\ 

{\em{Forward/backward adapting of measurement vectors}}\\

Starting  with $\wt_1^1=\frac{u_1}{||u_1||^2}$, and $u_1$ as in \eqref{defu},
the measurement vectors $ \wt_{i}^{k+1} ,i=1\ldots,k+1$ can be recursively 
constructed from $\wtik,\,i=1\ldots,k$ as follows \cite{Reb07a}:
\ben
\label{eq}
\wtikp&=&\wtik - \wt_{k+1}^{k+1} \la u_{k+1}, \wtik \ra,
\quad i=1,\ldots,k \label{wik}\\
\wt_{k+1}^{k+1}&=& \frac{\gamma_{k+1}}{||\gamma_{k+1}||^2},\quad
\gamma_{k+1}= u_{k+1}-\op_{\SWK}u_{k+1}, \label{wkkp}
\een
where $\op_{\SWK}$ is the orthogonal projector onto $\SWK=
\Span\{u_i\}_{i=1}^k$. We note that, since $\op_{\SW} \wtik= \wtik$, 
\eqref{wik} can also be written as
\ben
\wtikp&=&\wtik - \wt_{k+1}^{k+1} \la v_{k+1}, \wtik \ra,
\quad i=1,\ldots,k. \label{wik2}
\een
It follows from the above equations that when incorporating a linearly
independent atom 
$v_{k+1}$ in the atomic decomposition \eqref{ato}, the coefficients 
can be conveniently modified according to the recursive equations
\ben
\label{recf}
c_{k+1}^{k+1}&=& \la  \wt_{k+1}^{k+1} , f \ra, \\
c_i^{k+1}&= &\la \wtikp , f \ra= c_i^{k} - c_{k+1}^{k+1} \la \wtik, v_{k+1} \ra,
\quad i=1,\ldots,k. \label{wik22}
\een
Conversely, considering that the atom, $v_j$ say, is to be removed 
from the atomic decomposition \eqref{ato}, and 
denoting  the corresponding subspaces 
$\SVKJ$ and  $\SWKJ$, in  
order to span $\SWKJ$
the measurement vectors $\ut_i^{\kj},\,i=1,\ldots,k$ are modified 
according to the equation \cite{Reb07a}
\be
\label{duba}
\ut_i^{\kj}=\utik-\frac{\ut_j^k \la \ut_j^k , \utik\ra}{||\ut_j^k||^2},\quad
i=1,\ldots, j-1, j+1, \ldots, k.
\ee
Consequently, the coefficients in \eqref{ato} should be changed to
\be
\label{coba}
c_i^{\kj}=\cik-\frac{c_j^k \la \utik , \ut_j^k \ra}{||\ut_j^k||^2},\quad
i=1,\ldots, j-1, j+1, \ldots,k.
\ee
%Each time a new vector
%$v_{\ll_{k+1}}$ is added to the set ${\SV}_k$ to form the set ${\SV}_{k+1}$,
%we update the biorthogonal set
%$\{w^{k+1}_i,\,i=1\cdots,k+1\}$ by using the equations, see \cite{Reb06}:
%\begin{eqnarray}\label{dualupdate}
%w_{k+1}^{k+1}&=&\frac{\gamma_{\ll_{k+1}}}{\|\gamma_{\ll_{k+1}}\|^2};\;
%\gamma_{\ll_{k+1}}=u_{\ll_{k+1}}-P_{\SWK}u_{\ll_{k+1}},\\
%w_i^{k+1}&=&w_i^{k}-w_{k+1}^{k+1}\la u_{\ll_{k+1}},w_i^{k}\ra,\,i=1,\cdots,k,
%\end{eqnarray}
%where $P_{\SWK}$ is the orthogonal projection onto $\SWK$.
\section{Adaptive pursuit strategy for subspace selection} 
\label{sec4}
Given a signal $f$, we aim at finding the subspace $\SV_K \subset \SV$ 
where the signal belongs.  Let us stress once again that
the problem arises from the impossibility of correctly computing the 
measurement vectors spanning the whole subspace $\SW$.  
Otherwise the right subspace is determined simply by the 
indices corresponding to the atoms  having 
nonzero coefficients in the full atomic decomposition \eqref{ato}.

Since  $\SV_{k+1}+ \SWC =\SW_{k+1} \oplus \SWC$
the forward selection criterion we propose  is based on 
Property \ref{orcom}, which implies that if a given $f$
satisfies $\op_{\SW} f = \op_{\SWK} f$ it also satisfies $\EVW f = \EVKW f$. 
Thus, by fixing $\op_{\SWK}$,  at iteration $k+1$ we select the index 
$\ell_{k+1}$ such that 
$||\op_{\SW} f- \op_{\SWKP}f ||^2$ is minimized. 
\begin{proposition} 
Let us denote $J$ to the set of indices $i=1,\ldots,M.$
Given $\SWK$, the index $\ell_{k+1}$ corresponding the atom 
$u_{\ell_{k+1}}$ in the set $\{u_i\}_{i\in J}$
for which $||\op_{\SW} f- \op_{\SWKP}f||^2$ is minimal
is to be determined as
\be
\label{oomp}
 \ell_{k+1}= \arg\max\limits_{n \in J \setminus J_{k}}
\frac{|\la\gamma_n,f\ra|}{\|\gamma_n\|},\,\gamma_n \neq 0,
\ee
with $\gamma_n$ given in \eqref{wkkp}, and $J_{k}$ the set 
of indices that have been previously chosen to determine $\SWK$.
\end{proposition}
\begin{proof}
It readily follows  since $\op_{\SWKP} f= \op_{\SWK} f+ 
\frac{\gamma_n \la\gamma_n,f\ra}{\|\gamma_n\|^2}$ and hence
$||\op_{\SW} f - \op_{\SWKP} f ||^2=  ||\op_{\SW} f || - || \op_{\SWK} f ||^2 -
\frac{|\la \gamma_n,f \ra |^2}{\|\gamma_n\|^2}   $.  Because  
$\op_{\SW} f$ and $\op_{\SWK}f$ are
fixed, $||\op_{\SW} f - \op_{\SWKP} f||^2$ is minimized if 
$\frac{|\la\gamma_n,f\ra|}{\|\gamma_n\|},\,\gamma_n \neq 0$ is maximal 
over all $n\in J\setminus J_{k}$.
\end{proof}
\begin{remark} 
Since $\op_{\SWKP} = \op_{\SW}\EVWKP$ we can 
write
$||\op_{\SW} f - \op_{\SWKP} f ||^2= 
||\op_{\SW}(f -  \EVWKP f) ||^2$, and 
the condition of the previous proposition can be seen as the condition 
for minimizing the distance of $\EVWKP f$ to $f$, with 
respect to the weighted seminorm $\|\cdot\|_{P_{\SW}}$ induced by
the weighted inner product $\la\cdot,\cdot\ra_{P_{\SW}}$ defined as
$\la f, g\ra_{P_{\SW}}=\la f,{P_{\SW}}g\ra.$
\end{remark}
The OBMP selection criterion given in 
\cite{Reb07b}, which is based on the consistency principle  
 \cite{UA94,Eld03}, selects the index 
$\ell_{k+1}$ as the maximizer over $n \in J\setminus J_{k}$ of
$$\frac{|\la\gamma_{n},f\ra|}{\|\gamma_{n}\|^2},\quad ||\gamma_n||\neq 0.$$
This condition was proposed in \cite{Reb07b} so as to select the measurement 
vector  $\wtkp_{k+1}$
producing the maximum consistency error
$\Delta= |\la \wtkp_{k+1}, f - \EVKW f \ra|$, with regard to 
a new  measurement $\wtkp_{k+1}$.  However, since the measurement 
vectors are not normalized to unity, it is sensible to consider 
the consistency error relative to the corresponding 
measurement vector norm $||\wtkp_{k+1}||$, and select the index  
so as to maximize over $k+1 \in J \setminus J_{k}$ 
the relative consistency error 
\be
\label{reler}
\til{\Delta}= \frac{|\la \wtkp_{k+1}, f - \EVKW f \ra|}{||\wtkp_{k+1}||}, 
\quad ||\wtkp_{k+1}|| \ne 0
\ee
\begin{property}
The index $\ell_{k+1}$ satisfying \eqref{oomp}  maximizes over $k+1 \in J\setminus J_{k}$ the 
relative consistency error \eqref{reler}
\end{property}
\begin{proof}
Since for all vector $\wtkp_{k+1}$  given in 
\eqref{wkkp} $\EVKW \wtkp_{k+1}=0$ and $||\wtkp_{k+1}||= ||\gamma_{k+1}||^{-1}$ 
we have
$$\til{\Delta}= \frac{|\la \wtkp_{k+1}, f \ra |}{||\wtkp_{k+1}||}= 
\frac{|\la \gamma_{k+1}, f \ra |}{||\gamma_{k+1}||}.$$
Hence, maximization of $\til{\Delta}$ over $k+1 \in J\setminus J_{k}$ is equivalent 
to \eqref{oomp}. 
\end{proof}
It is clear at this point that the forward selection of 
indices prescribed by proposition \eqref{oomp} is equivalent to 
selecting the indices by applying 
the Optimized Orthogonal Matching Pursuit (OOMP) \cite{RL02}
strategy on the projected signal $\op_{\SW}f$ using the 
dictionary $\{u_i\}_{i\in J}$. 

The hypothesis that the computation of more than $r$ measurement 
vectors becomes an ill posed problem enforces the 
forward selection of indices to stop if iteration $r$ is reached.
Nevertheless, the fact that 
the signal is assumed to be $K$-sparse, with $K \le r$, does not
imply that before (or at) iteration $r$ one will always 
find the correct subspace. 
The $r$-value just indicates that it is not possible to 
continue with the forward selection, because 
the computations would become inaccurate and unstable. 
Hence, if the right solution was not yet found, one needs 
to implement a strategy  accounting for the fact that it is not feasible
to compute more than $r$ measurement vectors.  An 
adequate procedure is achieved by means of the swapping-based 
refinement to the OOMP approach introduced in 
\cite{AR06}. As discussed below, it consists of interchanging 
already selected atoms with nonselected  ones.

Consider that at iteration $r$ the correct subspace has not appeared yet
 and the selected indices are labeled  by the $r$
indices $\ell_1,\ldots,\ell_r$. 
In order to choose the label of the atom that minimizes the norm of the 
residual error as passing 
from approximation $\op_{\SWKr} f$ to approximation $\op_{\SWKrJ}f$ 
we should fix the index of the atom to be deleted, $\ell_j$ say,
as the one for which the quantity
\be
\label{swab}
\frac{|c_i^r|}{||w_i^r||}
\ee
is minimized $i=1,\ldots,r$ \cite{ARS04,AR06}.

The process of eliminating one atom from the atomic decomposition 
\eqref{ato} is called {\em backward step} while the  process of 
adding one atom is called {\em forward step}. 
The forward selection criterion to choose the atom to replace the 
one eliminated in the previous step is accomplished by 
 finding the index $\ell_i, i=1,\ldots,r$ for which the  
 the functional 
\be
\label{swaf}
e_n= \frac{|\la\nu_n , f \ra|}{||\nu_n||},\quad 
\text{with}\quad \nu_n =u_n - \op_{\SWKrJ} u_n,\quad ||\nu_n||\neq 0
\ee
is maximized.
In our framework, using \eqref{duba}, the projector 
$\op_{\SWKrJ}$ is computed as 
$$\op_{\SWKrJ} = \op_{\SWKr} - 
\frac{\la \ut_i^r , \ut_j^r\ra \la \ut_j^r, \cdot\ra}{||\ut_j^r||^2}.$$
Since $\op_{\SWKr}$ and $\ut_j^r$ are available, the 
computation of the sequence $\nu_n$ in \eqref{swaf} is a simple operation. 

As proposed in \cite{AR06} the swapping of pairs of atoms is 
repeated until the swapping operation, if carried out, would not 
decrease the approximation error. The implementation details 
for an effective realization of this process are given in \cite{AR06}, 
and MATLAB codes are available at \cite{Webpage}.  
Since there is no guarantee that at the end of the swapping 
of pairs of atoms the correct subspace has been found, the
process can continue by increasing the number of 
atoms the swapping involves. At the second stage,  in line with  
 \cite{AR06a} we propose the swapping 
to be realized by the combinations of two backward steps followed by 
two forward steps, provided that the interchange of the two 
atoms improves the approximation error. If at the end of the
second stage the right subspace has not yet been found, the  number of 
atoms involved in the swapping is increased up to three 
and so on. Notice that 
if the number of atoms to be interchanged reaches the value $r$ the 
whole process would repeat identically. This is avoided by 
initiating the new circle with a different initial atom. 
Although convergence cannot be guaranteed, the above specified  
hypothesis ensure that the algorithm will stop when the  correct 
signal splitting has been found. At such a stage one has
$\op_{\SW} f = \op_{\SW_r}f$ with $\SW_r$ spanned by the selected atoms 
$\ell_1,\ldots,\ell_r$. If the order $K$ of sparseness of the signal is 
less than $r$ a number of $r-K$  coefficients in the atomic 
decomposition 
$$f=\sum_{i=1}^r v_{\ell_i} \la w_i^r,f \ra = 
\sum_{i=1}^r c_i^r v_{\ell_i}$$
will have zero value.
\subsection{Examples}
Firstly we applied the proposed strategy to the numerical simulation 
of Example 1, which is a very simple test for our method and therefore 
in a run of 50 simulations we could produce the correct signals splitting  
at the stage involving forward selection only.\\

{\bf{Example 2.}} For this example we have used the same background as
in Example 1, but a dictionary of B-splines  spanning the same space 
as the basis. The dictionary consists of functions of broader 
support than the basis functions for the same space, and the 
translation parameter is reduced (for more details on the construction 
of B-splines dictionaries see \cite{AR05}, MATLAB codes 
are available at \cite{Webpage}). In this case, 
the spectrum of singular values of matrix $G$ decreases continuously, 
as shown in the top graphs of Figure 2. Since it is difficult to 
decide on where to truncate the singular values, 
for the sake of comparison with the proposed technique we 
made a signal dependent truncation. 
This was  achieved  by setting the number $Q$ of  singular values to be 
considered so as to minimize
$$||\op_{\SW} f  - \op_{\til{\SW}_Q} f||,$$
where  ${\til{\SW}_Q}$ indicates the subspace spanned by the first 
$Q$ singular vectors of the operator $\hat{W}$ (c.f. \eqref{xi})
Neither in this case the regularization by truncation of singular values  
was successful. The result is 
depicted by the lighter line in the right bottom graph of Figure 2. The 
dark line plots the sought signal. 

\begin{figure}[!ht]
%\label{f2}
\begin{center}
\includegraphics[width=7.5cm]{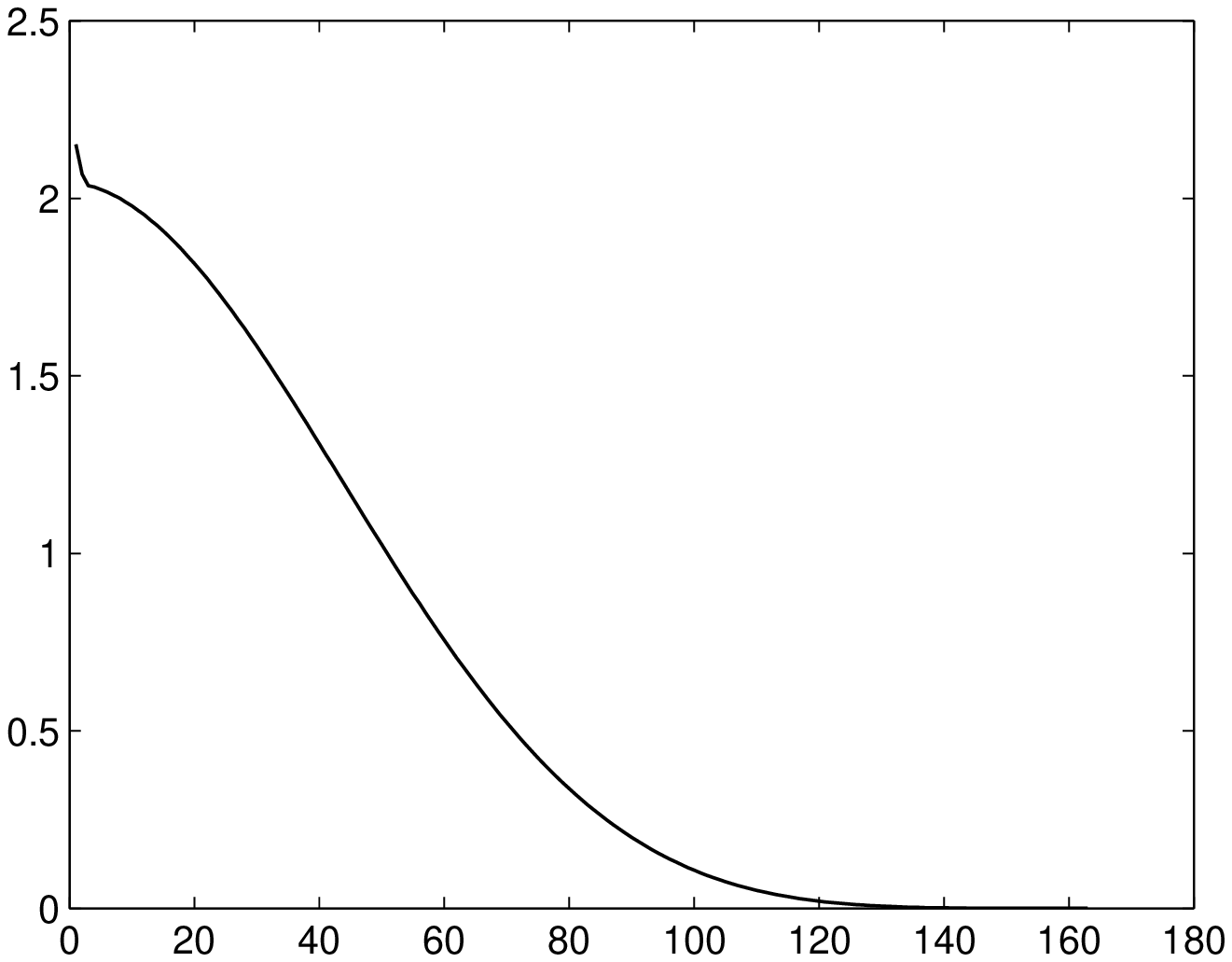}
\includegraphics[width=7.5cm]{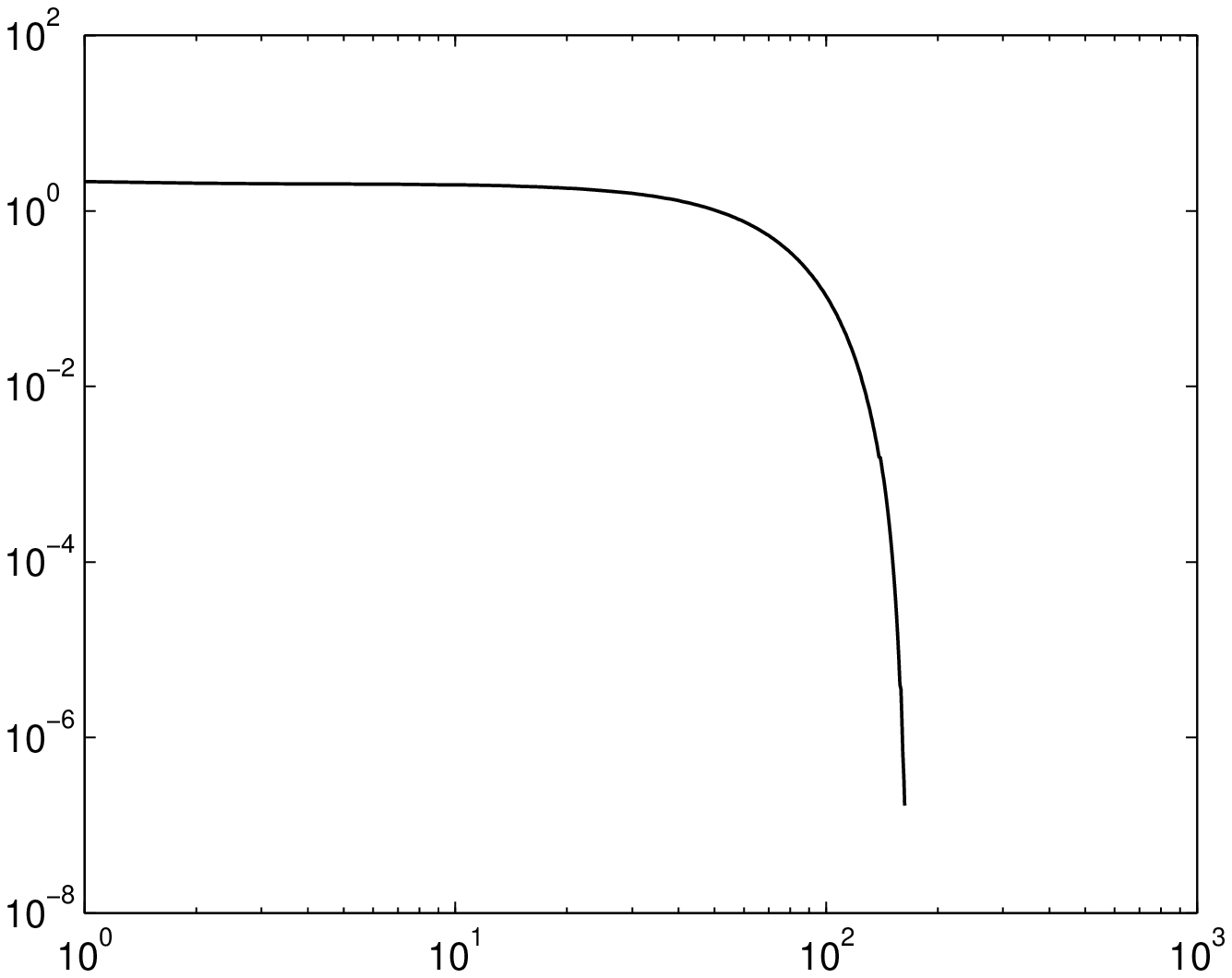}\\
\includegraphics[width=7.5cm]{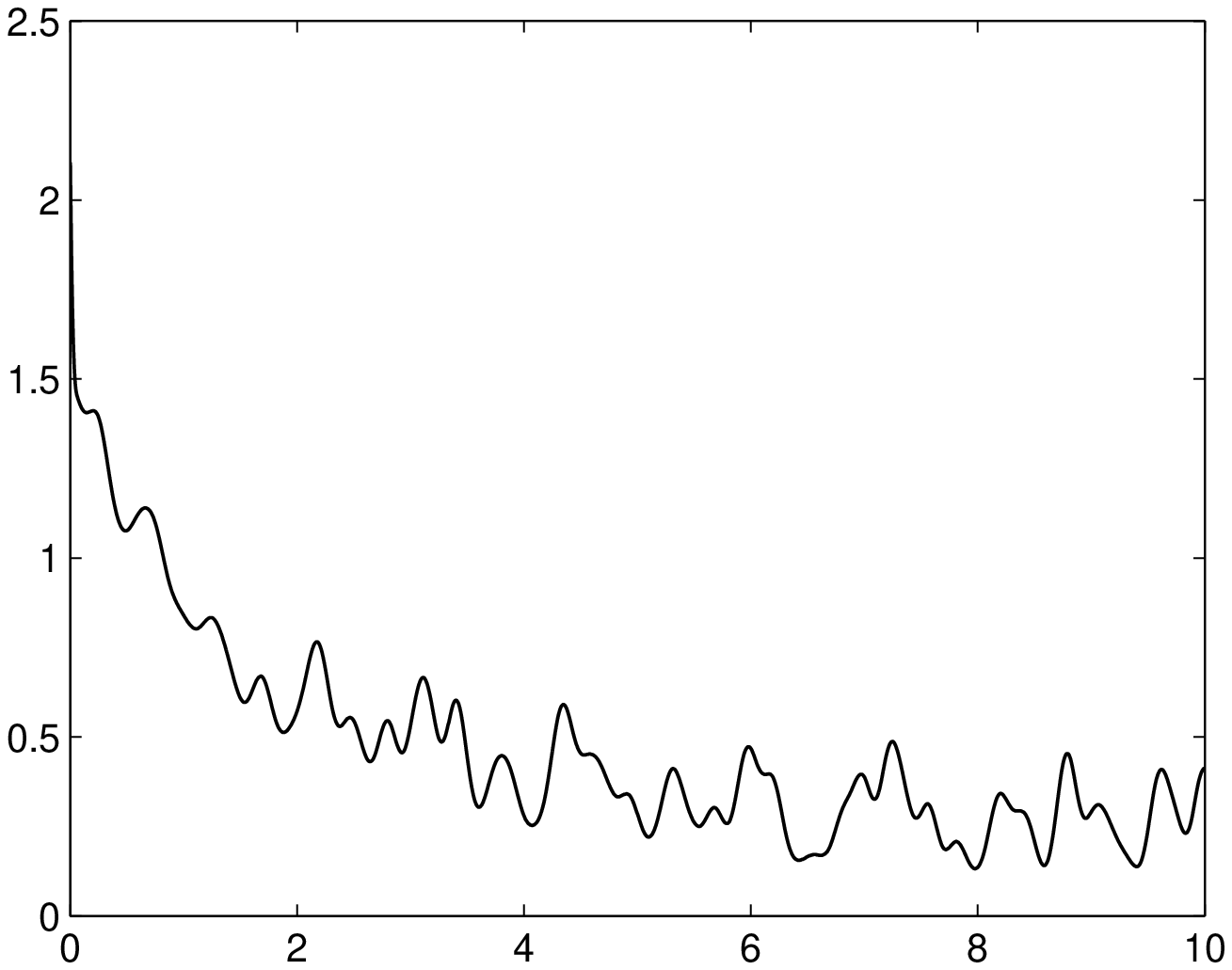}
\includegraphics[width=7.5cm]{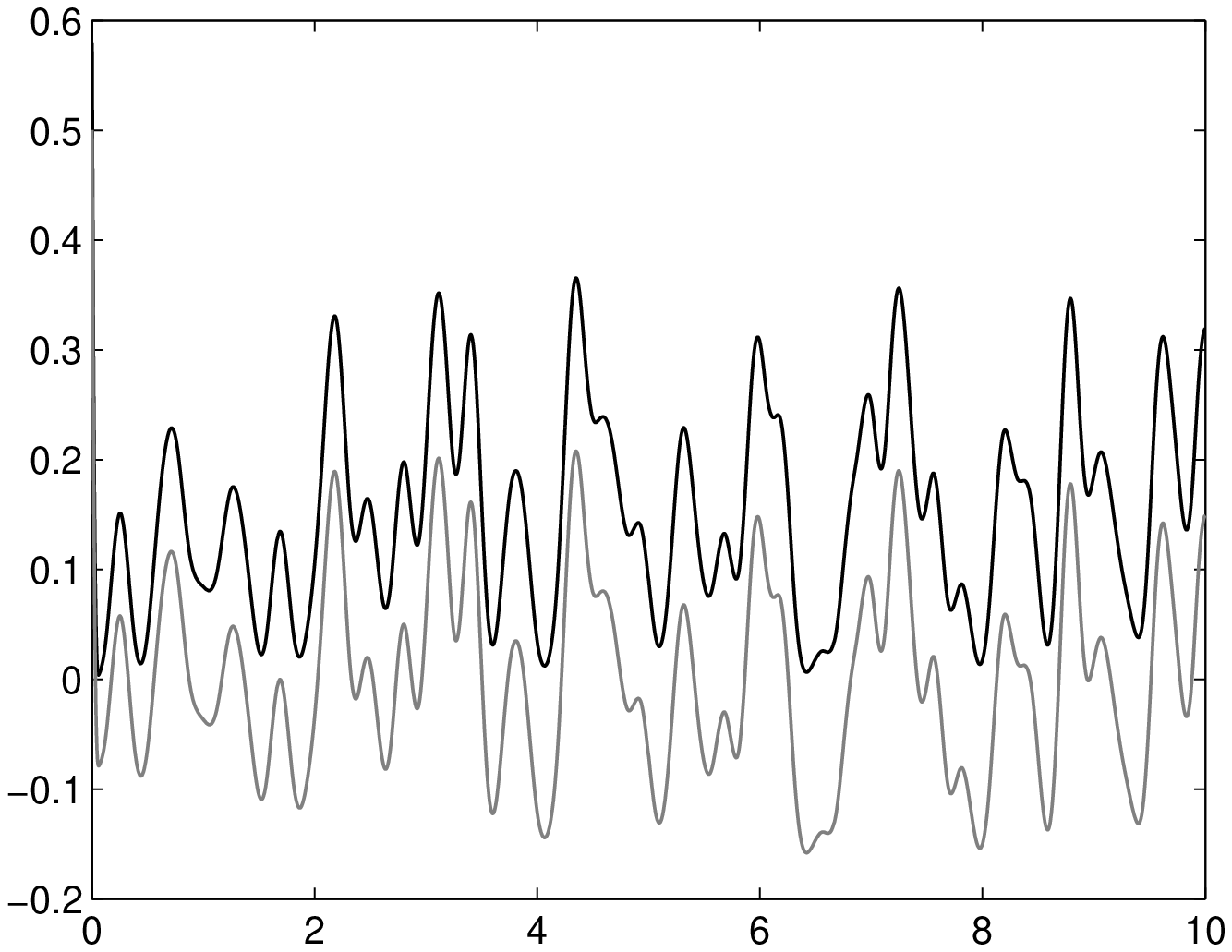}\\
\end{center}
\caption{Top graphs: spectrum of singular values of 
matrix $G$ in Example 2. The right graph plots the line in the 
left graph, but in logarithmic scale.
\newline
Bottom graphs: The one on the left shows the signal plus background.
The dark line in the right one is the signal component to be discriminated 
from the signal in the left graph. The light line represents the 
approximation obtained by truncation of singular values. The proposed 
approach reproduces exactly the dark line.}  
\end{figure}

By means of the proposed strategy we were able to find the 
correct signal splitting in the 50 simulations we ran. 
Only in one of the cases a re-initialization took place. 
Increasing the number of atoms in the simulated atomic decomposition up 
to 80, we also found the correct splitting in all the cases. In this 
simulation the re-initialization stage occurred in the case of 5 
signals, out of the 50  signal in the run. By increasing the number
of atoms up to 90, re-initialization took place in the 
case of 8 signals. As could be 
expected, due to the  singular values decay, by increasing the 
number of atoms up to $100$ in the atomic decomposition we started to 
observe instability in the calculations.

{\bf{Example 3.}} 
Here the signal space
is spanned by $M=210$ vectors in $R^{L},$ with $L=1000$, given 
as
\[\{v_i=\cos(\frac{\pi(2j-1)(i-1)}{2L}),\quad j=1,\ldots,L\}_{i=1}^M.\]
The space of the noise is spanned by the set
\[\{y_i=e^{{-35000(j-0.005i)}^2}, \quad j=1,\ldots,L\}_{i=1}^{400}.\]
We ran 50 simulations, keeping the noise fixed and 
considering a different realization of the signal, which 
was generated as a linear combination of $90$ vectors 
taken randomly from the given spanning set. 

The spectrum of singular values is depicted in the top graphs 
of Figure 3. In this case the signal dependent criterion for 
truncation does achieve the correct signal splitting. The
  left bottom graph of Figure 3  shows one of the 
realizations of the signal plus noise in the simulation. 
The dark line in the right graph of Figure 3 plots the
exact signal. The approximation obtained by truncation of singular values
is plotted with a lighter line,  which cannot be distinguished from the
dark one in the scale of the figure.
It is clear from this result that in this case the signal does not have
a significant component in the subspace spanned by the
neglected singular vectors. Similar results are obtained
in all the  other realizations in the simulation. The norm of the error
in this case is $0.53$, while the mean  value of the error norm with
respect to the 50 cases is $0.78$.

By applying the proposed strategy  for searching the  
 sparse representation we found the exact solution in the 50 cases. 
Re-initialization was necessary in 5 of the 50 realizations.

\begin{figure}[!ht]
\label{f3}
\begin{center}
\includegraphics[width=7.5cm]{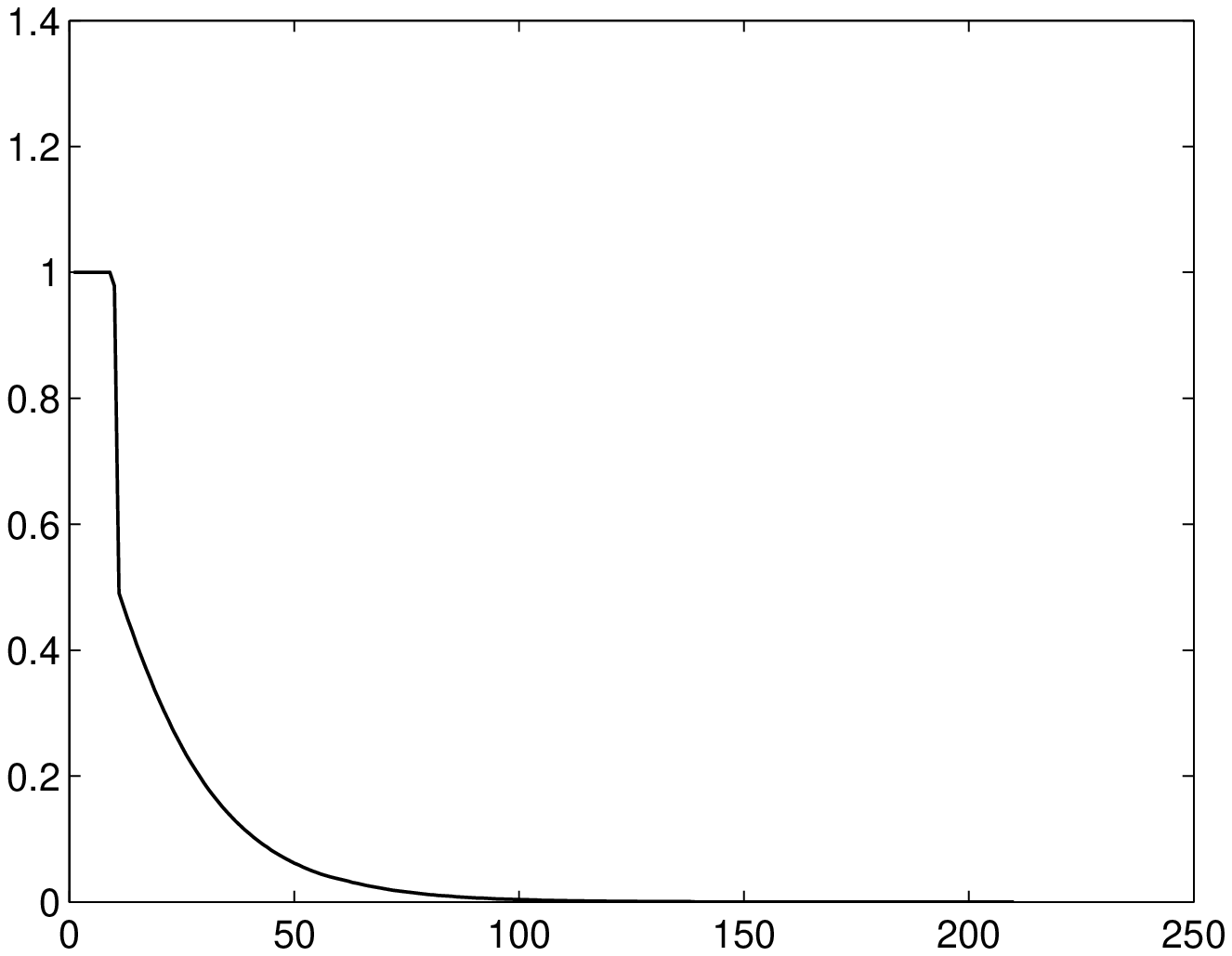}
\includegraphics[width=7.5cm]{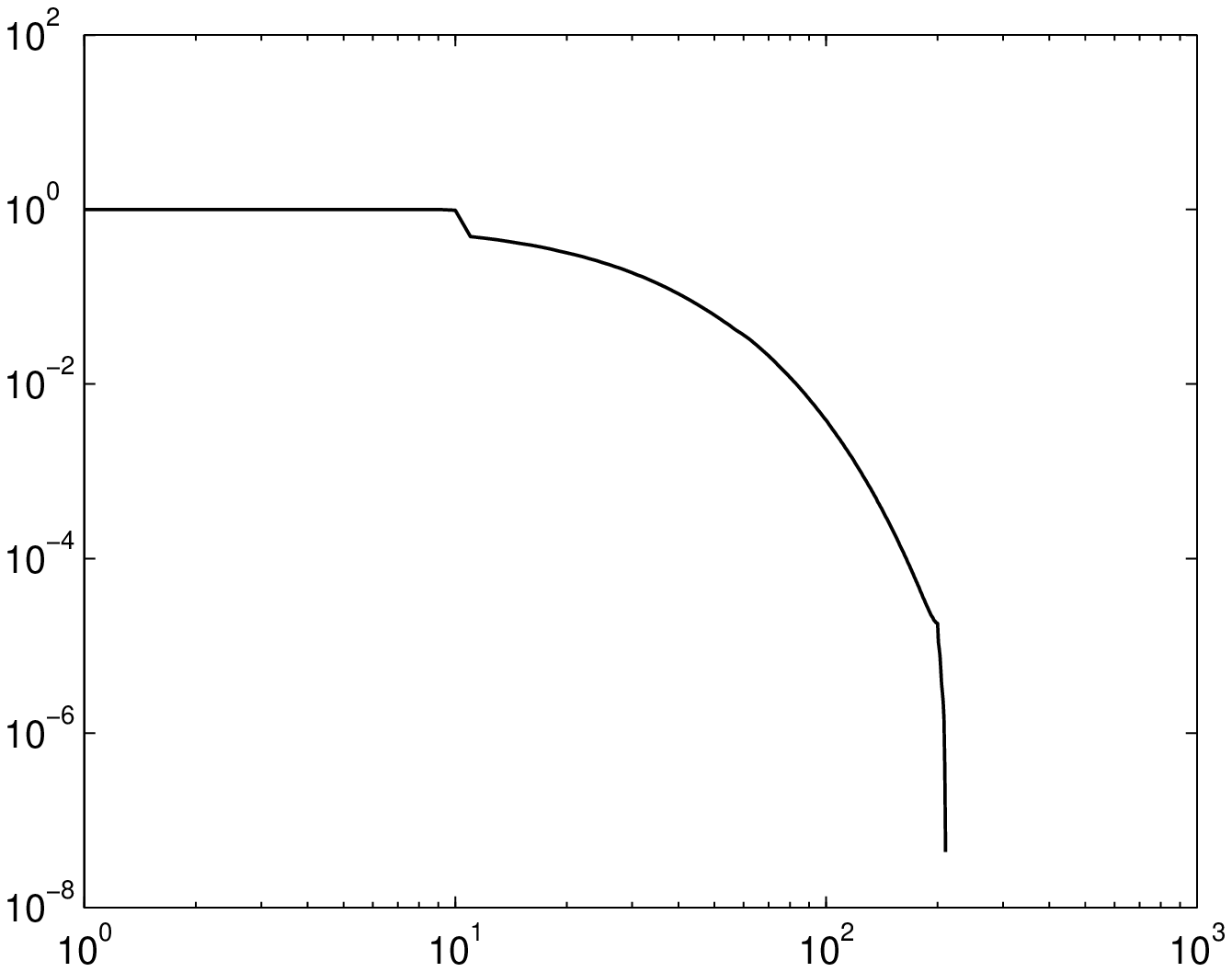}\\
\includegraphics[width=7.5cm]{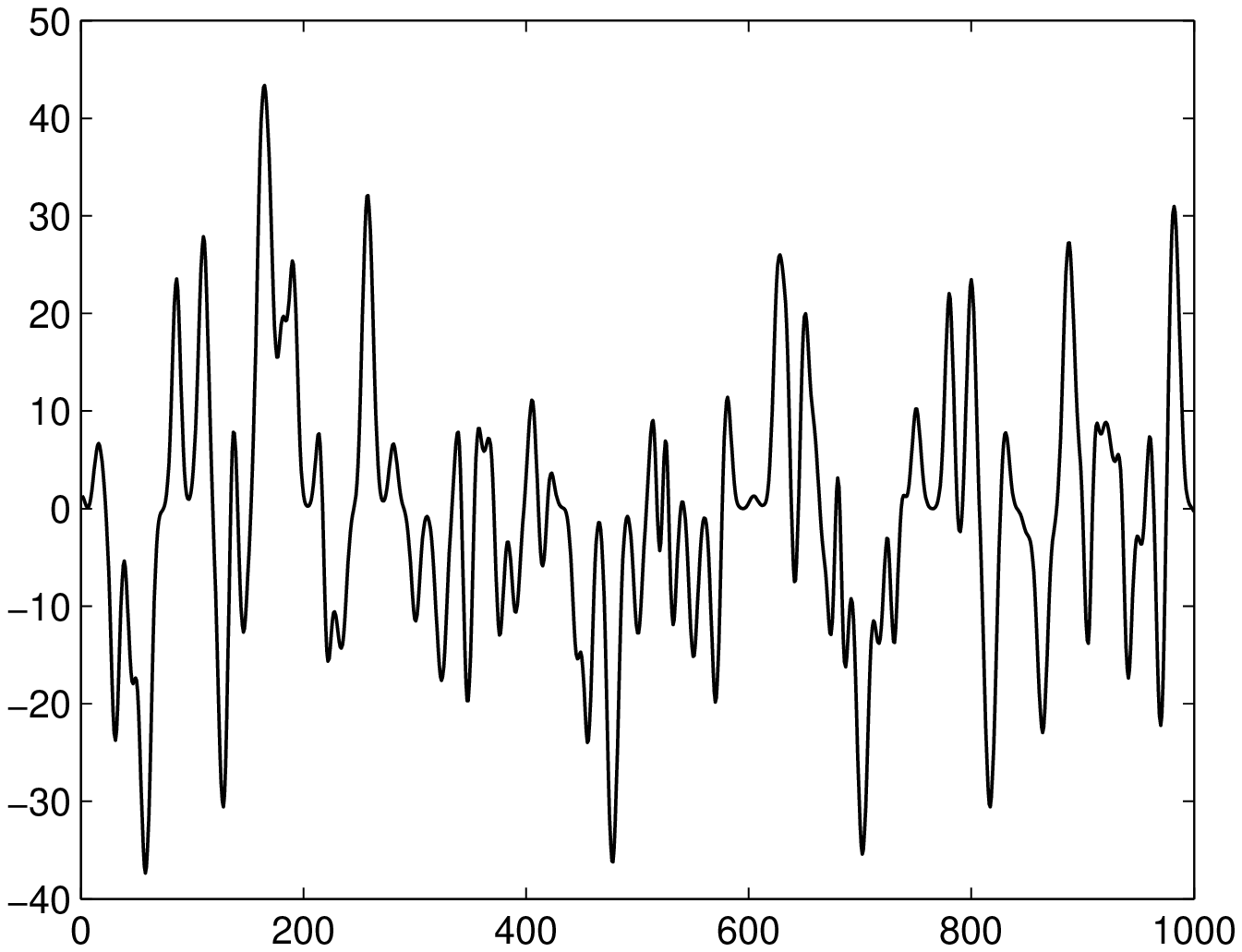}
\includegraphics[width=7.5cm]{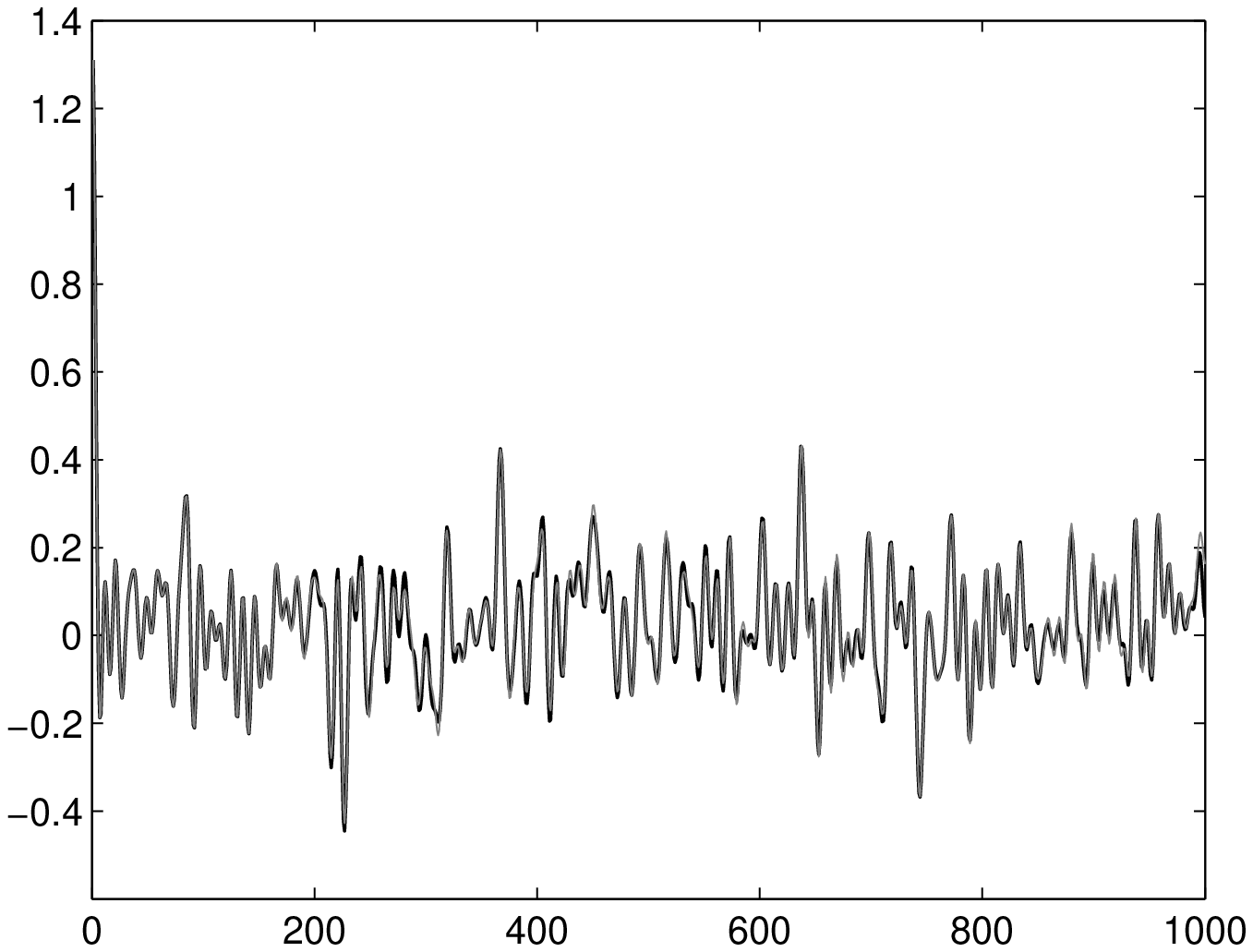}
\end{center}
\caption{Top graphs: spectrum of singular values of
matrix $G$ in Example 3. The right graph plots the line in the
left graph, but in logarithmic scale.\newline
Bottom graphs: The one on the left shows the signal plus noise.
The dark line in the right one is the signal to be separated from 
the one in the left graph. The approximation obtained by 
truncation of singular values is plotted by a lighter line, 
which cannot be distinguished from the other in the scale of the 
figure. The proposed approach reproduces the exact signal.} 
\end{figure}

%However, if the signal is simulated as a  random linear superposition 
%of 160 atoms, the problem becomes much harder. Notice that 160 is 
%the maximum number of atoms for which the measurement vectors can be 
%calculated accurately and the dimension of the dictionary is 163. Thus 
%one can have a maximum of 3 wrong atom in the atomic decomposition, 
%even if they where chosen randomly.

\section{Conclusion}
\label{sec5}
The role of sparse representations in the context of structured noise 
filtering has been discussed. The discrimination of signal components 
is achieved by an oblique projection onto the right subspace.
Considerations were restricted to 
 those cases for which the signal subspace and the noise subspace are 
theoretically complementary, but the construction of the dual basis 
for the whole signal subspace yields an ill posed problem, due to 
the calculations being carried out in finite precision arithmetic. 
It was shown by numerical simulations that, if the signal is sparse
in a spanning set for the signal subspace, the required signal splitting 
may be achieved by means of adaptive techniques capable of 
searching for the required subspace while maintaining stability 
in the calculations. Although convergence of the proposed strategy for 
adaptive subspace search is not guaranteed, the method is capable to stop 
when the correct signal splitting is accomplished. 

For the sake of comparison, an alternative regularization technique 
based on adaptive truncation of singular 
values  was  analyzed. The main disadvantage of such a technique  lies in 
the fact that regularization is performed  by a change 
of subspaces. Consequently, in general,  
the technique does not produce the required signal splitting.  Moreover,
even when a satisfactory  splitting is attained (c.f. Example 3)
the method does not 
provide an indication that this is  so. Except for 
the very particular case in which  the  signal at hand 
has zero projection onto the  subspace 
spanned by the disregarded singular vectors, the exact solution cannot 
be produced by this technique. On the contrary, the approach based on 
the search for the sparse representation is capable of producing 
the exact solution when the method stops. Thus, if the algorithm has 
converged, one can assert that the signal splitting is the required one.

\bibliographystyle{elsart-num}
\bibliography{signal}
\end{document}